\pdfoutput=1
\documentclass[11pt]{article}
\usepackage[utf8]{inputenc}
\def\final{0}
\usepackage{color}
\usepackage{subfigure}
\usepackage{url}
\usepackage{graphicx}
\usepackage{amsmath}
\usepackage{amsthm}
\usepackage{bm}
\usepackage[plainpages=false]{hyperref}

\usepackage{titlesec}
\titleformat{\section}
{\normalfont\Large\bfseries\filcenter}{\thesection.}{1 ex}{}
\titleformat{\subsection}[runin]
{\normalfont\normalsize\bfseries\filcenter}{\thesubsection.}{1 ex}{}

\usepackage[charter]{mathdesign}
\usepackage[mathcal]{eucal}

\usepackage[margin=1.2 in]{geometry}

\usepackage[square,numbers,sort&compress]{natbib}

\ifnum\final=0
\newcommand{\mynote}[1]{\marginpar{\tiny\sf #1}}
\else
\newcommand{\mynote}[1]{}
\fi

\usepackage{theomac}
\newtheoremWithMacro{theorem}{Theorem}
\newtheoremWithMacro{lemma}{Lemma}[section]

\newtheoremWithMacro{corollary}{Corollary}
\newtheoremWithMacro{prop}[lemma]{Proposition}
\newtheoremWithMacro{conj}[theorem]{Conjecture}
\newtheoremWithMacro{question}[theorem]{Question}

\newcommand{\figref}[1]{Figure \ref{fig:#1}}

\newcommand{\propref}[1]{Proposition \ref{prop:#1}}
\newcommand{\proprefX}[1]{\ref{prop:#1}}
\newcommand{\theoref}[1]{Theorem \ref{theo:#1}}
\newcommand{\theorefX}[1]{\ref{theo:#1}}
\newcommand{\secref}[1]{Section \ref{sec:#1}}

\newcommand{\proplab}[1]{\label{prop:#1}}
\newcommand{\theolab}[1]{\label{theo:#1}}
\newcommand{\seclab}[1]{\label{sec:#1}}
\newcommand{\conjlab}[1]{\label{conj:#1}}

\renewcommand{\vec}[1]{\mathbf{#1}}

\newcommand{\iprod}[2]{\left\langle {#1},{#2}\right\rangle}

\newcommand{\R}{\mathbb{R}}
\newcommand{\Z}{\mathbb{Z}}

\newcommand{\eqlab}[1]{\label{eq:#1}}
\renewcommand{\eqref}[1]{(\ref{eq:#1})}

\newcommand{\eop}{\hfill$\qed$}

\newcommand{\bgamma}{\bm{\gamma}}

\newcommand{\cent}{\operatorname{cent}}
\newcommand{\teich}{\operatorname{teich}}
\newcommand{\teichgamma}{\operatorname{teich_{\Gamma}}}
\newcommand{\N}{\mathbb{N}}
\newcommand{\Euc}{\operatorname{Euc}}
\newcommand{\Rep}{\operatorname{Rep}}
\newcommand{\Teich}{\operatorname{Teich}}
\newcommand{\ttmat}[4]{\begin{pmatrix}{#1} & {#2} \\ {#3} & {#4}\end{pmatrix}}
\newcommand{\JMat}[1]{ \left( \begin{array}{rr} #1 \end{array} \right)}
\newcommand{\Trans}{\Lambda}
\DeclareMathOperator{\rk}{rk}
\DeclareMathOperator{\PSL}{PSL}

\begin{document}
\title{Generic rigidity with forced symmetry and sparse colored graphs}
\author{Justin Malestein\thanks{Department of Mathematics,
Temple University, \url{justmale@temple.edu}}
\and Louis Theran\thanks{Institut für Mathematik,
Diskrete Geometrie, Freie Universität Berlin, \url{theran@math.fu-berlin.de}}}
\date{}
\maketitle
\begin{abstract}
We review some recent results in the generic rigidity theory of planar frameworks with forced symmetry,
giving a uniform treatment to the topic.  We also give new combinatorial characterizations of minimally rigid
periodic frameworks with fixed-area fundamental domain and fixed-angle fundamental domain.
\end{abstract}

\section{Introduction} \seclab{intro}
The Maxwell-Laman Theorem is the touchstone result of combinatorial rigidity theory.
\begin{theorem}[\maxwelllaman][{\cite{M64,L70}}]
A generic bar-joint framework in the plane is minimally rigid if and only if the graph
defined by the frameworks edges has $n$ vertices $m=2n - 3$ edges, and, for all subgraphs
on $n'$ vertices and $m'$ edges, $m'\le 2n' -3$.
\end{theorem}
The key feature of this, and all such ``Laman-type results'' is that,
for almost all \emph{geometric} data, rigidity is determined by the \emph{combinatorial type} and can be
decided by efficient \emph{combinatorial} algorithms.

\subsection{Some generalizations}
Finding generalizations of the Maxwell-Laman Theorem has been the motivation for a lot of progress in the
field.  The \emph{body-bar} \cite{T84}, \emph{body-bar-hinge} \cite{T84,W88}, and \emph{panel-hinge} \cite{KT11}
frameworks has a rich generic theory in all dimensions.  Here the ``sparsity counts'' are of the form $m' \le Dn' - D$,
where $D$ is the dimension of the $d$-dimensional Euclidean group.  On the other hand, various elaborations of the
planar bar-joint model via \emph{pinning} \cite{F07,L80,R89}, \emph{slider-pinning} \cite{ST10,KT11a},
\emph{direction-length frameworks} \cite{SW99}, and other geometric restrictions like
\emph{incident vertices} \cite{FJK11} or, of more relevance here
\emph{symmetry} \cite{S10,S10a}, have all shed more light on the Maxwell-Laman Theorem itself.

In another direction, various families of graphs and hypergraphs defined by \emph{heriditary sparsity counts}
of the form $m'\le kn' - \ell'$ have been studied in terms of \emph{combinatorial structure} \cite{LS08},
\emph{inductive constructions} \cite{LS08,FS07}, \emph{sparsity-certifying decompositions} \cite{ST09,W88} and \emph{linear
representability} \cite{ST11} \cite[Appendix A]{W96} properties.  Running through much of this work is a matroidal perspective first
introduced by Lovász-Yemini \cite{LY82}.

While a lot is known about $(k,\ell)$-sparse graphs and hypergraphs, the parameter settings that yield interesting
\emph{rigidity} theorems seem to be somewhat isolated, despite the uniform combinatorial theory and
many operations that move between different sparsity families.

\subsection{Forced symmetry}
For the past several years, the rigidity and flexibility of frameworks with additional \emph{symmetry}
has received much attention,%
\footnote{See, e.g., the recent conferences \cite{LMS10,F11,RS12}.}
although it also goes back further.
Broadly speaking, there are two approaches to this: \emph{incidental symmetry}, in which one studies
a framework that may move in unrestricted ways but starts in a symmetric position \cite{S10,S10a,KG00,FG00,KG99,OP10}; and \emph{forced}
symmetry \cite{BS10,MT10,R09,MT11,MT12,RSW11} where a framework must maintain symmetry with respect to a
specific group throughout its motion.  Forced symmetry is particularly useful as a way to study \emph{infinite} frameworks%
\footnote{Infinite frameworks with no other assumptions can exhibit quite complicated behavior \cite{OP11}.}
arising in applications to crystallography \cite{R06,TRBR04}.

In a sequence of papers \cite{MT10,MT11,MT12}, we developed the generic rigidity theory for
the forced-symmetric frameworks in the plane.  The basic setup we consider is as follows: we are given a group
$\Gamma$ that acts discretely on the plane by Euclidean
isometries, a graph $\tilde{G} = (\tilde{V},\tilde{E})$, and a $\Gamma$-action $\varphi$ on $\tilde{G}$
with finite quotient that is free on the vertices and edges.  A (realized) \emph{$\Gamma$-framework}
$\tilde{G}(\vec p,\Phi)$ is given by a point set $\vec p = (\vec p_i)_{i\in \tilde{V}}$ and a
representation $\Phi$ of $\Gamma$ by Euclidean isometries, with the compatibility
condition
\begin{equation}\eqlab{compatibility}
\vec p_{\gamma\cdot i} = \Phi(\gamma)\cdot \vec p_i
\end{equation}
holding for all $\gamma\in \Gamma$ and $i\in \tilde{V}$.

Intuitively, the allowed continuous motions through $\tilde{G}(\vec p,\Phi)$
are those that preserve the lengths and connectivity of the bars, and symmetry with
respect to $\Gamma$, but the particular representation $\Phi$ is allowed to flex.  When
the only allowed motions are induced by Euclidean isometries, a framework is \emph{rigid}, and otherwise
it is \emph{flexible}.

The combinatorial model for $\Gamma$-frameworks is \emph{colored graphs},
which we describe in \secref{maxwell}.  These efficiently capture some
canonical $\Gamma$-framework invariants relating to how much flexibility from the
group representation $\Phi$ a sub-framework constraints.  The state of the art is:
\begin{theorem}[\stateoftheart][{\cite{MT10,MT11,MT12}}]\theolab{bigtheorem}
Let $\Gamma$ be one of:
\begin{itemize}
\item $\Z^2$, acting on the plane by translation
\item $\Z/k\Z$, for $k\in \N$, $k\ge 2$ acting on the plane by an order $k$ rotation around the origin
\item $\Z/2\Z$, acting on the plane by a reflection
\item A crystallographic group generated by translations and a rotation.
\end{itemize}
A generic $\Gamma$-framework $\tilde{G}(\vec p,\Phi)$ is minimally rigid if and only if the
associated colored quotient graph $(G,\bgamma)$ has $n$ vertices, $m$ edges and:
\begin{itemize}
\item $m = 2n + \teichgamma(\Gamma) - \cent(\Gamma)$
\item For all subgraphs $G'$ on $n'$ vertices, $m'$ edges, with connected components $G_i$ that
have $\rho$-image $\Gamma_i$,
\begin{equation}\eqlab{sparsity}
m'\le 2n' + \teichgamma(\Lambda(G')) - \sum_{i}\cent(\Gamma_i')
\end{equation}
where $\Lambda(G')$ is the translation subgroup associated with $\Gamma'_i$.
\end{itemize}
\end{theorem}
\noindent (See \secref{maxwell} for definitions of $\teichgamma$ and $\cent$.)
\theoref{bigtheorem} gives a generic rigidity theory that is:
(1) Combinatorial; (2) Computationally tractable; (3) Applicable to \emph{almost all} frameworks;
(4) Applicable to a small \emph{geometric} perturbation of \emph{all} frameworks.  In other words,
it carries \emph{all} of the key properties of the Maxwell-Laman-Theorem to the forced symmetry setting.

\subsection{Results and roadmap}  The classes of colored graphs appearing in \theoref{bigtheorem}
are a new, non-trivial, extension of the $(k,\ell)$-sparse families that had not appeared
before.  The proof of \theoref{bigtheorem} relies on a \emph{direction network method} (cf. \cite{ST10,W88}),
and the papers \cite{MT10,MT11,MT12} develop the required combinatorial theory for direction networks.  In this paper,
we focus more on frameworks, describing the colored graph invariants that correspond to ``Maxwell-type heuristics''
and showing how to explicitly compute them.  Additionally, we study periodic frameworks in a bit more detail,
and derive several new consequences of \theoref{bigtheorem}: conditions for a periodic framework to fix
the representation of $\Z^2$ (\propref{lattice-fix-implies-rk2-block}, \propref{unit-area-maxwell}), and,
as a consequence, the Maxwell-Laman-type \theoref{unit-area} for periodic frameworks with fixed area
fundamental domain.

\subsection{Notation and terminology}
We use some standard terminology for $(k,\ell)$-sparse graphs: a finite graph $G=(V,E)$ is
\emph{$(k,\ell)$-sparse} if for all subgraphs on $n'$ vertices and $m'$ edges, $m'\le kn'-\ell$.
If equality holds for all of $G$, then $G$ is a \emph{$(k,\ell)$-graph}; a subgraph for which
equality holds is a \emph{$(k,\ell)$-block} and maximal $(k,\ell)$-blocks are
\emph{$(k,\ell)$-components}.  Edge-wise minimal violations of $(k,\ell)$-sparsity are
$(k,\ell)$-circuits.  If $G$ contains a $(k,\ell)$-graph as a spanning subgraph it is
\emph{$(k,\ell)$-spanning}. A \emph{$(k,\ell)$-basis} of $G$ is a maximal subgraph that is $(k,\ell)$-sparse.
We refer to $(2,3)$-sparse graphs by their more conventional name: Laman-sparse graphs.

In the sequel, we will define a variety of hereditarily sparse \emph{colored} graph families.  We generalize
the concepts of ``\emph{sparse}'', ``\emph{block}'', ``\emph{component}'', ``\emph{basis}'' and ``\emph{circuit}'' in the
natural way for any family of colored graphs defined by a sparsity condition.

\subsection{Acknowledgements}
We thank the Fields Institute for its hospitality during the Workshop on
Rigidity and Symmetry, the workshop organizers for putting together the program,
and the conference participants for many interesting discussions. LT is supported by
the European Research Council under the European Union's Seventh Framework Programme (FP7/2007-2013)
/ ERC grant agreement no 247029-SDModels. JM is supported by NSF CDI-I grant DMR 0835586.

\section{The model and Maxwell heuristic}\seclab{maxwell}
We now briefly review the degree of freedom heuristic that leads to the sparsity
condition \eqref{sparsity}.  As is standard, we begin with the desired form:
\begin{equation}\eqlab{generic-heuristic}
\#(\text{constraints}) \le \#(\text{total d.o.f.}) - \#(\text{trivial motions})
\end{equation}
What distinguishes the forced symmetric setting is that the r.h.s. depends, in an essential
way, on the representation of $\Phi$ the symmetry group.  Thus, we modify \eqref{generic-heuristic} to
\begin{equation}\eqlab{symmetric-heuristic}
\#(\text{constraints}) \le \#(\text{total non-trivial d.o.f.}) -
\#(\text{rigid motions preserving $\Phi$})
\end{equation}

\subsection{Flexibility of symmetry groups and subgroups}
Let $\Gamma$ be a group as in \theoref{bigtheorem}.  We define the \emph{representation space} $\Rep(\Gamma)$
to be the set of all faithful representations $\Phi$ of $\Gamma$ by Euclidean isometries.
The \emph{Teichmüller space}%
\footnote{We are extending the terminology ``Teichm\"uller space''
from its more typical usage for the group $\Z^2$ and lattices in $\PSL(2, \R)$.  Our
definition of $\Teich(\Z^2)$ is non-standard since the usual one allows only unit-area fundamental domains.}
$\Teich(\Gamma)$ is the quotient $\Rep(\Gamma)/\Euc(2)$ of the representation
space by Euclidean isometries.  We define $\teich(\Gamma)$ to be the dimension of $\Teich(\Gamma)$.  For
frameworks, the Teichmüller space plays a central role, since $\teich(\Gamma)$ gives the total number of
\emph{non-trivial} degrees of freedom associated with representations of $\Gamma$.

Now let $\Gamma' < \Gamma$ be a subgroup of $\Gamma$.  The
\emph{restricted Teichmüller space} $\Teich_{\Gamma}(\Gamma')$ is the image of the
restriction map from $\Gamma\to \Gamma'$ modulo Euclidean isometries. Equivalently
it is the space of representations of $\Gamma'$ that extend to representations of $\Gamma$.
Its dimension is defined to be $\teichgamma(\Gamma)$.

The invariant $\teichgamma(\Gamma')$ measures how much of the flexibility of $\Gamma$ can be
``seen'' by $\Gamma'$.
In general, the restricted Teichmüller space of $\Gamma'$ is \emph{not}
the same as its (unrestricted) Teichmüller space.  For instance, the Teichm\"uller space $\Teich(\Z^2)$ has
dimension $3$, but the restricted Teichm\"uller space $\Teich_\Gamma(\Z^2)$ has dimension $1$
if $\Gamma$ contains a rotation of order $3$.

\subsection{Isometries of the quotient}
Now let $\Phi$ be a representation of $\Gamma$.  The \emph{centralizer} of $\Phi$ is the subgroup
of Euclidean isometries commuting with $\Phi$.  We define $\cent(\Gamma)$ to be the dimension of
the centralizer, which is independent of $\Phi$ (see e.g. \cite[Lemma 6.1]{MT11}).  An alternative
interpretation of the centralizer is that it is the isometry group of the
quotient orbifold $\R^2/\Gamma$.

\subsection{Colored graphs}\seclab{lift}
The combinatorial model for a $\Gamma$-framework is a \emph{colored graph} $(G,\bgamma)$ %
\footnote{Colored graphs are also known as ``gain graphs'' or ``voltage graphs'' \cite{Z98}.
The  terminology of colored graphs originates from \cite{R06}.}%
, which is a finite, directed graph $G=(V,E)$ and an assignment
$\bgamma = (\gamma_{ij})_{ij\in E}$ of a group element in $\Gamma$ to each edge of $G$.
The correspondence between colored graphs $(G,\bgamma)$ and graphs with a $\Gamma$-action
$(\tilde{G},\varphi)$ is a straightforward specialization of covering space theory,
and we have described the dictionary in detail in \cite[Section 9]{MT11}.  The important
facts are:
\begin{itemize}
\item The data $(\tilde{G},\varphi)$ and a selection of a representative from each vertex and edges
determine a colored graph $(G,\bgamma)$.
\item Each colored graph $(G,\bgamma)$ lifts to a graph $\tilde{G}$ with a free $\Gamma$-action by a
natural construction.
\end{itemize}
Together these mean that the colored graph $(G,\bgamma)$ captures \emph{all} the information
in $(\tilde{G},\varphi)$.

\subsection{The homomorphism $\rho$}
Let $(G,\bgamma)$ be a connected colored graph, and select a base vertex $b$ of $G$.  The coloring on the edges then
induces a natural homomorphism $\rho : \pi_1(G,b)\to \Gamma$.  For a closed path $P$ defined by the sequence of
edges $b i_2, i_2i_3, \dots, i_{\ell-1} b$, we have
$$\rho(P) = \gamma_{b i_2} \gamma_{i_2 i_3} \dots \gamma_{i_{\ell-1} b}.$$
The key properties of $\rho$ are \cite[Lemmas 12.1 and 12.2]{MT11}:
\begin{itemize}
\item The quantities $\teichgamma(\rho(\pi_1(G,b)))$ and
$\cent(\rho(\pi_1(G,b)))$ depend only on the lift $(\tilde{G},\varphi)$,
so, in particular, they are independent of the choice of $b$.
\item If $G_1, G_2,\ldots, G_c$, are the connected components of a \emph{disconnected} colored graph
$(G,\bgamma)$, there is a well-defined \emph{translation subgroup} $\Lambda(G)$ of $G$.
\end{itemize}

\subsection{Derivation of the Maxwell heuristic}
We are now ready to derive the degree of freedom heuristic for $\Gamma$-frameworks.
Let $(G,\bgamma)$ be a $\Gamma$-colored graph with $n$ vertices, $m$ edges, connected
components $G_1,G_2,\ldots, G_i$, with $\rho$-images $\Gamma'_i$.  We fill in the template
\eqref{symmetric-heuristic} for the associated $\Gamma$-framework $\tilde{G}(\vec p,\Phi)$:
\paragraph{Non-trivial degrees of freedom}
There are two sources of flexibility the group representation $\Phi$ and the coordinates
of the vertices.
\begin{itemize}
\item[\textbf{(A)}] The group representation $\Phi$ has, by definition, $\teichgamma(\Lambda(G))$
degrees of freedom, up to Euclidean isometries.  These are the non trivial degrees of $\Phi$ ``seen'' by $G$.
\item[\textbf{(B)}] The coordinates of the vertices are determined by the location of
one representative of each $\Gamma$-orbit and $\Phi$.  There are $n$ such
orbits, for a total of $2n$ degrees of freedom
\end{itemize}
Here is the guiding intuition for \textbf{(A)}.
We want to understand is how the edge lengths can constrain the representation $\Phi$.  It is
intuitively clear that if there is no pair of points $\tilde{\vec p_i}$
and $\tilde {\vec p_{\gamma \cdot i}}$ in the same $\Gamma$-orbit that are also connected
in the lift $(\tilde{G},\varphi)$, the the framework cannot constrain $\Phi$ at all.
Thus, we are interested in accounting for constraints arising from paths in $(\tilde{G},\varphi)$
between pairs of points $\tilde{\vec p_i}$ and  $\tilde {\vec p_{\gamma \cdot i}}$;
in $(G,\bgamma)$, this corresponds to a closed path $P$ with $\rho(P)=\gamma$.

This reasoning leads us to consider $\teich_\Gamma(\cdot)$ of a subgroup generated by
the $\rho$-images of some closed paths in $(G,\bgamma)$.  After some technical analysis,
the correct subgroup is discovered to be $\Trans(G)$.

\paragraph{Rigid motions independent of $\Phi$}
For each connected component of $\tilde{G}(\vec p,\Phi)$ induced by $G_i$: there is a
$\cent(\Gamma'_i)$-dimensionalspace of these for each $G_i$, since any element of the centralizer of
$\Gamma'_i$ preserves all the edge lengths and compatibility with $\Phi$.  Because the components are disconnected,
these motions are independent of each other.

\section{Periodic frameworks}\seclab{periodic}
A $\Gamma$-framework with symmetry group $\Z^2$ is called a \emph{periodic framework} \cite{BS10}.  In this
section, we specialize \eqref{sparsity} to this case, and relate it to an alternate counting
heuristic from \cite[Section 3]{MT10}.

\subsection{Invariants for $\Z^2$}
Representations of $\Z^2$ by translations have very simple coordinates: they are given by mapping each of
the generators $(1,0)$ and $(0,1)$ to a vector in $\R^2$.  Thus, the space of (possibly degenerate)
representations is isomorphic to the space of $2\times 2$ matrices with real entries.  Given such a
matrix $\vec L$ and $\gamma\in \Z^2$, the translation representing $\gamma$ is simply $\vec L\cdot \gamma$.
Because of this identification, we denote realizations of periodic frameworks by $\tilde{G}(\vec p,\vec L)$,
and call $\vec L$ the \emph{lattice representation}.

Subgroups of $\Z^2$ are always generated by $k=0,1,2$ linearly independent vectors; given a
subgroup, we define its \emph{rank} to be the minimum size of a generating set.
To specify a representation of a subgroup $\Gamma'<\Z^2$, we assign a vector in $\R^2$ to each of the
$k$ generators of $\Gamma'$.  Such a representation always extends to a faithful representation of $\Z^2$.
Thus, we see that dimension of the space of representations of $\Z^2$ restricted to $\Gamma'$ is $2k$.

The quotient of the representation space of $\Z^2$ by $\Euc(2)$ is also straightforward to describe.  Each
point has a representative $\vec L$ such that $\vec L\cdot (1,0) = (\lambda,0)$ for some real
scalar $\lambda$.  From this, we get:
\begin{prop}\proplab{z2-teich}
Let $\Gamma'<\Z^2$ be a subgroup of $\Z^2$ with rank $k$.  Then $\teich_{\Z^2}(\Gamma') = \max \{2k-1,0\}$.
\end{prop}

Finally, we compute the dimension of the centralizer of a subgroup $\Gamma'$.  If $\Gamma'$ is trivial, then
the centralizer is the entire $3$-dimensional Euclidean group.  If $\Gamma'$ is rank $1$, then it is represented
by a translation $t_1(\vec p)=\vec p + \vec t_1$, which commutes with other translations and reflections or glides
fixing a line in the direction $\vec t_1$.  For the rank $2$ case, the centralizer is just the translation subgroup of
$\Euc(2)$.  We now have:
\begin{prop}\proplab{z2-cent}
Let $\Gamma'<\Z^2$ be a subgroup of $\Z^2$ with rank $k$.  Then
\[
\cent(\Gamma') = \begin{cases}
3 & \text{if $k$ = 0} \\
2 & \text{if $k\ge 1$}
\end{cases}
\]
\end{prop}

\subsection{The homomorphism $\rho$ for $\Z^2$}
Now we turn to associating a colored graph $(G,\bgamma)$ with a subgroup of $\Z^2$.
This is simpler than the general case because $\Z^2$ is abelian, so we may define it as a
map $\rho : H_1(G,\Z)\to \Z^2$, as was done in \cite{MT10}.  Here are the relevant facts:
\begin{prop}\proplab{rho-for-z2}
Let $(G,\bgamma)$ be a colored graph.  Then the rank of the $\rho$-image is
determined by the values of $\rho$ on \emph{any} homology (alternatively, cycle) basis
of $G$, and thus $\rho$ is well-defined when $G$ has more than one connected component.
\end{prop}

\subsection{Colored-Laman graphs}
With Propositions \proprefX{z2-teich}--\proprefX{rho-for-z2}, the colored graph sparsity counts
\eqref{sparsity} from \theoref{bigtheorem} specializes, for a $\Z^2$-colored graph to:
\begin{equation}\eqlab{z2-sparsity}
m' \le 2n' + \max\{2k-1,0\} - 3c'_0 - 2c'_{\ge 1}
\end{equation}
where $k$ is the rank of the $\Z^2$-image of $(G,\bgamma)$, $c'_0$ is the number of connected components with trivial
$\Z^2$-image and $c'_{\ge 1}$ is the number of connected components with non-trivial $\Z^2$-image (i.e., $k\ge 1$).
This gives a matroidal family \cite[Lemma 7.1]{MT10}, and we define the bases to be \emph{colored-Laman graphs}.

\subsection{An alternative sparsity function}
A slightly different counting heuristic for a periodic framework with colored quotient graph $(G,\bgamma)$
having $n$ vertices, $m$ edges, $c$ connected components and $\rho$-image with rank $k$ is as follows:
\begin{itemize}
\item There are $2n$ variables specifying the points, and $2k$ variables giving a representation of the $\rho$-image.
\item To remove Euclidean isometries  that move the points and the lattice representation together, we
pin down a connected component.
\item Each of the remaining connected components may translate independently of each other.
\end{itemize}
Adding up the degrees of freedom and subtracting three degrees of freedom for pinning down one connected components
and two each for translations of each other connected component yields the sparsity condition from \cite[Section 3, p. 14]{MT10}
\begin{equation}\eqlab{z2-alt-sparsity}
m' \le 2(n'+k) - 3 - 2(c'-1)
\end{equation}
which is equivalent to the colored-Laman counts \eqref{z2-sparsity} by the following.
\begin{prop}\proplab{z2-sparsity-equiv}
Let $(G,\bgamma)$ be a $\Z^2$-colored graph.  Then $(G,\bgamma)$ satisfies \eqref{z2-sparsity} if and only
if it satisfies \eqref{z2-alt-sparsity}.
\end{prop}
\begin{proof}
For convenience, we define the two functions:
\begin{eqnarray}
\eqlab{z2-f-fun}	f(G) & = & 2n +  \max\{2k-1,0\} - 3c'_0 - 2c'_{\ge 1} \\
\eqlab{z2-g-fun}	g(G) & = & 2(n+k- c) - 1
\end{eqnarray}
where $g$ is easily seen to be equal to the r.h.s. of \eqref{z2-alt-sparsity}.  The
definitions imply readily that $f(G)\le g(G)$, with equality when there is either
one connected component in $G$ or all connected components have $\rho$-images with
rank at least one.  Thus, it will be sufficient to show that, if $(G,\bgamma)$, has $n$ vertices, $m$ edges,
and $\rho$-image of rank $k$, and it is minimal with the property that
$f(G) = m - 1$, then $g(G) = m - 1$.

Let $(G,\bgamma)$ have these
properties, and let $G$ have connected components $G_i$ with, $n_i$ vertices, $m_i$ edges, and
$\rho$-images of rank $k_i$.
The minimality hypothesis implies that for any $G_i$, the number of edges
in $G\setminus G_i$ is
\begin{equation}
\eqlab{sparsity-eqv-pf}
m - m_i  \le  f(G\setminus G_i)
\end{equation}
but, if $k_i$ is zero, the rank of the $\rho$-image of $G\setminus G_i$ is $k$,
and $m_i \le 2n_i - 3$. Computing, we find that
\begin{align*}
m - m_i  & \ge 2n +  \max\{2k-1,0\} - 3c'_0 - 2c'_{\ge 1} + 1 - 2n_i + 3\\
&  = 2(n-n_i) + \max\{2k-1,0\} - 3(c'_0 - 1) - 2c'_{\ge 1} + 1\\
&  = f(G\setminus G_i) + 1
\end{align*}
which is a contradiction to \eqref{sparsity-eqv-pf}.
We conclude that either there is one connected component in $G$
or that none of the $k_i$ were zero.  In either of these cases $f(G)=g(G)$, which completes
the proof.
\end{proof}

\subsection{Example: Disconnected circuits}
The proof of \propref{z2-sparsity-equiv} generalizes the folklore fact that, for Laman rigidity,
we get the same class of graphs from ``$m'\le 2n' -3$'' and the more precise ``$m'\le 2n' - 3c'$''.
In the periodic setting the additional precision is \emph{required}:
\begin{itemize}
\item There are periodic frameworks with \emph{dependent} edges in different connected components of the colored
quotient graph \cite[Figure 20]{MT10}.
\item There are \emph{connected} $\Z^2$-colored graphs that are not colored-Laman sparse but satisify
\eqref{z2-sparsity} for \emph{all} induced or connected subgraphs \cite[Figure 8]{MT10}.
\end{itemize}
The intuition leading to the discovery of \eqref{z2-alt-sparsity} is that connected components of a
periodic framework's colored graph interact via the representation $\vec L$ when they have the same
$\rho$-image.

\subsection{Example: Disconnected minimally rigid periodic frameworks}
Another phenomenon associated with periodic rigidity that is not seen in finite
frameworks is that although the colored quotient graph $(G,\bgamma)$ must be
connected \cite[Lemmas 4.2 and 7.3]{MT10}, the periodic framework $\tilde{G}(\vec p,\Phi)$
does not need to be as in \cite[Figure 9]{MT10}.  To see this, we simply note that
\eqref{z2-sparsity} depends only on the \emph{rank} of the $\rho$-image, which is unchanged
by multiplying the entries of the colors $\gamma_{ij}$ on the edges $(G,\bgamma)$ by an integer $q$.
On the other hand, this increases the number of connected components by a factor of $q^2$.
There is not paradox because periodic symmetry is being forced:
once we know the realization of one connected component of $\tilde{G}(\vec p,\Phi)$, we can
reconstruct the rest of them from the representation $\Phi$ of $\Z^2$.

\subsection{Conditions for fixing the lattice}\seclab{periodic-lattice-fix}
The definition of rigidity for periodic frameworks implies that a rigid framework fixes
the representation $\vec L$ of $\Z^2$ up to a Euclidean isometry.  It then follows that
\emph{any} periodic framework with a non-trivial \emph{rigid component} must do the same.
However, this is not the only possibility.  \figref{lattice-fix-not-rigid} shows a framework
without a rigid component that fixes the lattice representation and its associated colored graph.
The framework's non-trivial motion is a rotation of each of
the triangles.  This example is part of a more general phenomenon.
\begin{figure}[htbp]
\centering
\subfigure[]{\includegraphics[width=.3\textwidth]{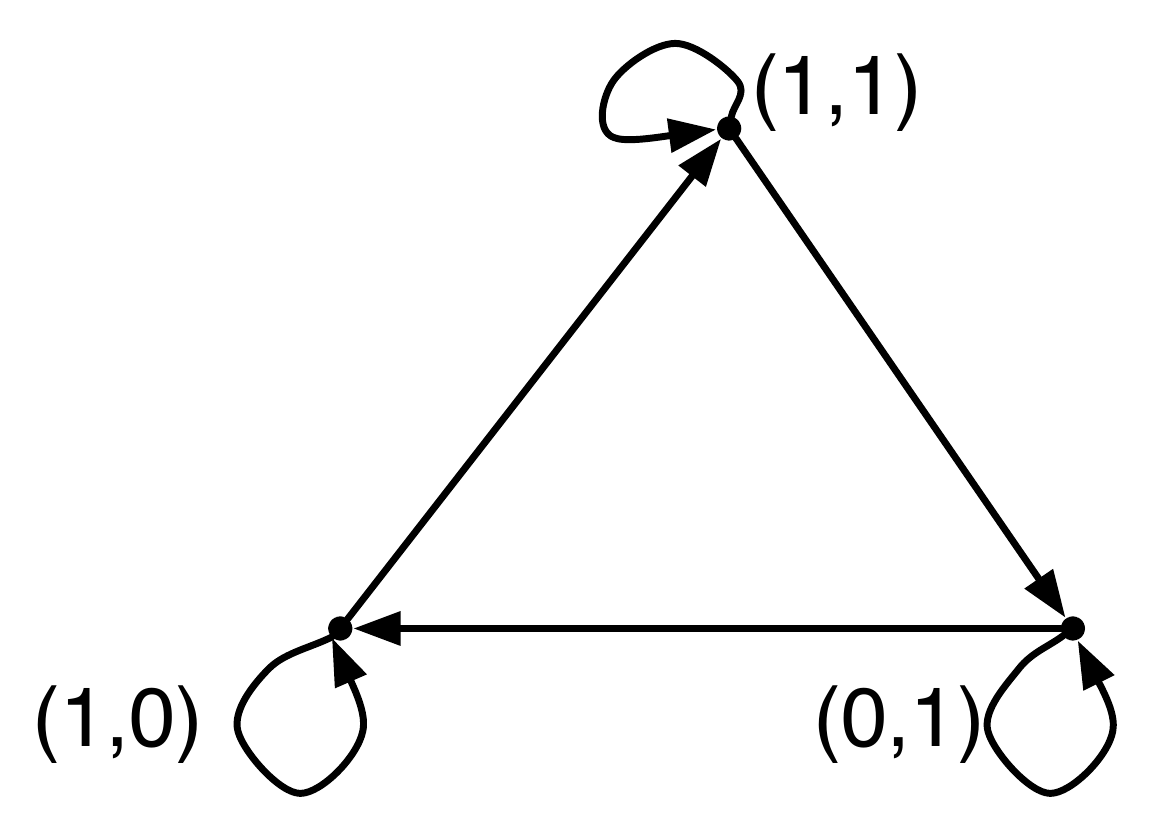}}
\subfigure[]{\includegraphics[width=.3\textwidth]{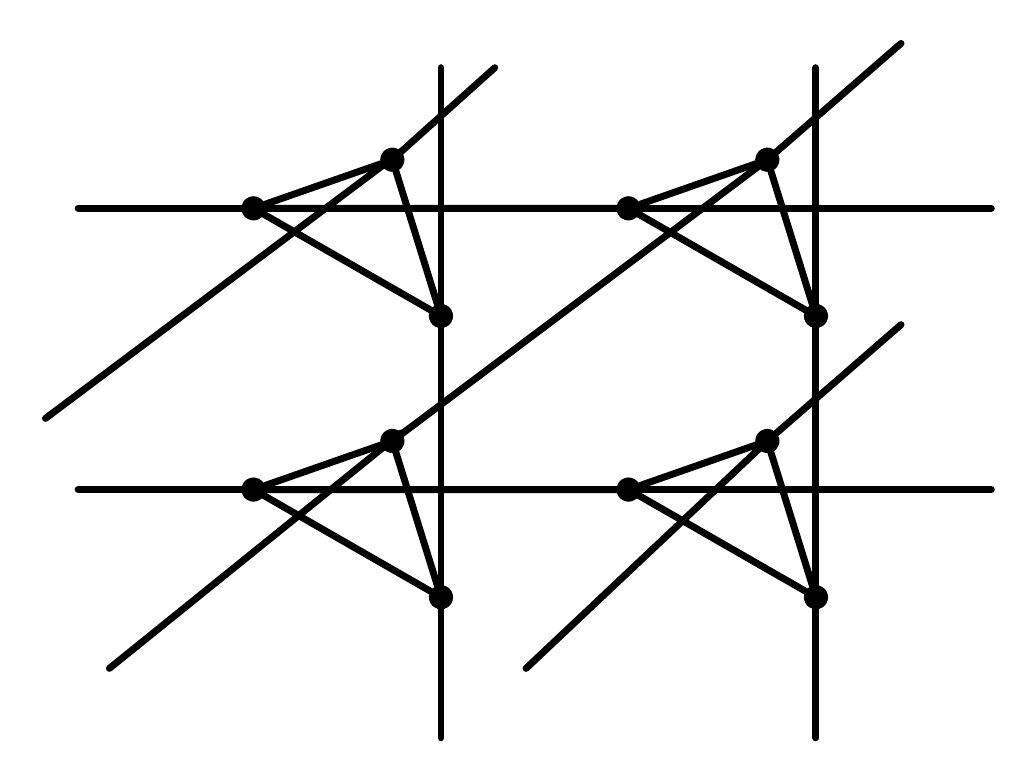}}
\caption{A flexible periodic framework that determines the lattice representation:
(a) the associated colored graph; (b) the periodic framework.}
\label{fig:lattice-fix-not-rigid}
\end{figure}
\begin{prop}\proplab{lattice-fix-implies-rk2-block}
Suppose that $(G,\bgamma)$ is a colored graph such that an associated generic
periodic framework $\tilde{G}(\vec p,\Phi)$ fixes the lattice representation.
Then $(G,\bgamma)$ contains a subgraph $G'$ with $m$ edges and rank $2$ $\rho$-image
such that $m = f(G')$, where $f$ is the sparsity function defined in \eqref{z2-f-fun}.
\end{prop}
\begin{proof}
We may assume without loss of generality that $(G,\bgamma)$ is colored-Laman sparse.
Let $\eta\in \Z^2$ be a vector that is linearly independent of any $\rho$-image of any rank $1$ subgraph of $(G,\bgamma)$.
(Such $\eta$ exists since there are only finitely many induced such lines in $\Z^2$.)
It follows from \theoref{bigtheorem} and the hypothesis that $\tilde{G}(\vec p,\Phi)$ is generic and
fixes the lattice representation that adding a self-loop $\ell$ with
color $\eta$ leads to a colored graph that is not colored-Laman-sparse.  This implies that there is a
minimal subgraph $(G'+\ell,\bgamma)$ of $(G+\ell,\bgamma)$ that is not colored-Laman sparse.  This
$G'$ must contain $\ell$.  The $\rho$-image of $G'$ must be rank $2$ because, if it were not, the rank of
the $\rho$-image of $G'+\ell$ would be strictly larger and thus not violate the sparsity condition.
The graph $G'$ is the subgraph of $(G,\bgamma)$ required by the statement of the proposition.
\end{proof}

\section{Specializations of periodic frameworks}\seclab{specializations}
Because \theoref{bigtheorem} is quite general, we can  deduce
Laman-type theorems for many restricted versions of periodic frameworks
from \theoref{bigtheorem}.  In this section, we describe three of these in
detail and discuss connections with some others.

\subsection{The periodic rigidity matrix}
The proof of \theoref{bigtheorem} relies on giving a combinatorial characterization
of \emph{infinitesimal rigidity} with forced symmetry constraints. The
\emph{rigidity matrix}, which is the formal differential of the length equations
plays an important role.  For periodic frameworks, this has the following form, which was
first computed in \cite{BS10}:
\begin{equation}\eqlab{rigidity-matrix}
\bordermatrix{                   &              &     i       &             &     j      &            &    L_1   & L_2  \cr
& \dots & \dots & \dots & \dots & \dots & \dots  & \dots \cr
ij   &  \dots & -\eta_{ij} & \dots & \eta_{ij} & \dots & \gamma_{ij}^1 \eta_{ij} & \gamma_{ij}^2 \eta_{ij} \cr
& \dots & \dots & \dots & \dots & \dots & \dots & \dots  }
\end{equation}
Here $\eta_{ij} = (\vec p_j + \vec L\cdot \gamma_{ij} - \vec p_i)$ is the vector describing a
representative of an edge orbit in $\tilde{G}(\vec p,\vec L)$, which we identify with a
colored edge of the quotient $(G,\bgamma)$.  There is one row for each edge in the quotient graph $G$.
The column groups $L_1$ and $L_2$ correspond to the
derivatives with respect to the variables in the rows of
$\vec L = \begin{pmatrix} a & b \\ c & d\end{pmatrix}$.  A framework is \emph{infinitesimally rigid}
if the rigidity matrix has corank $3$, and \emph{infinitesimally flexible} with $d$ degrees of freedom
if the rigidity matrix has corank $3+d$.
A framework is \emph{generic} if the rank of the rigidity matrix is maximal over
all frameworks with the same colored quotient graph.
We will require some standard facts about infinitesimal rigidity that transfer from the finite
to the periodic setting.
\begin{prop}\proplab{infinitesimal-rigidity-fact}
Let $\tilde{G}(\vec p,\vec L)$ be a periodic framework with quotient graph $(G,\bgamma)$.  Then:
\begin{itemize}
\item For generic frameworks, infinitesimal rigidity and flexibility coincide with rigidity and flexibility \cite{BS10,MT10}.
\item Infinitesimal rigidity and flexibility are affinely invariant \cite{BS10}, with non-trivial infinitesimal motions
mapped to non-trivial infinitesimal motions.
\end{itemize}
\end{prop}

\subsection{One flexible period}\seclab{cylinder}
A very simple restriction of the periodic model is to consider frameworks with one flexible
period.  The symmetry group is then $\Z$, acting on the plane by translations; we call such a framework a \emph{cylinder framework}.
We model the situation with $\Z$-colored graphs, and a single vector $\vec l\in \R^2$ representing the period lattice.
In this case, the $\rho$-image of a colored graph always has rank zero or one.

We define a \emph{cylinder-Laman graph} to be a $\Z$-colored graph $(G,\bgamma)$ such that:
$G$ has $n$ vertices, $2n - 1$ edges, and satisfies,
for all subgraphs, on $n'$ vertices, $m'$ edges, $\rho$-image of rank $k$,
$c'_0$ connected components with trivial $\rho$-image, and
$c'_1$ connected components with non-trivial $\rho$-image:
\begin{equation}\eqlab{cylinder-sparsity}
m'\le 2n' + k - 3c'_0 - 2c'_1
\end{equation}
Comparing \eqref{cylinder-sparsity} with \eqref{z2-sparsity}, we see readily:
\begin{prop}\proplab{cylinder-laman-sparsity}
The family of cylinder-Laman graphs corresponds bijectively with the maximal colored-Laman sparse graphs that
have colors of the form $\gamma_{ij} = (\cdot,0)$.
\end{prop}
\begin{theorem}\theolab{cylinder}
A generic cylinder framework is minimally rigid if and only if its associated colored graph is
cylinder-Laman.
\end{theorem}
\begin{proof}
The rigidity matrix for a cylinder framework has the same form as \eqref{rigidity-matrix}, except with the
column group labeled $L_2$ discarded.  \propref{cylinder-laman-sparsity} and then \theoref{bigtheorem}
yield the desired statement.
\end{proof}

\subsection{Unit area fundamental domain}\seclab{unit-area}
Next, we consider the class of \emph{unit-area frameworks}, for which the allowed motions preserve the
area of the fundamental domain of the $\Z^2$-action on the plane induced by the $\Z^2$-representation
$\vec L$.

We define a \emph{unit-area-Laman graph} to be a $\Z^2$-colored graph $(G,\bgamma)$ with $n$
vertices, $m = 2n$ edges, and satisfying, for all subgraphs on $n'$ vertices, $m'$ edges,
and $c'_k$  connected components with $\rho$-image of rank $k$
\begin{eqnarray}\eqlab{unit-area-sparsity1}
m' & \le  2n' - 1 - 3c_0 - 2(c'_1 - 1) & \qquad\text{if $c'_2 = 0$} \\
\eqlab{unit-area-sparsity2}
m' & \le  2n' - 3c_0 - 2(c'_1 + c'_2 - 1) & \qquad\text{if $c'_2 > 0$}
\end{eqnarray}

\begin{theorem}\theolab{unit-area}
A generic unit-area framework is minimally rigid if and only if its associated colored graph
$(G,\bgamma)$ is unit-area-Laman.
\end{theorem}
\paragraph{Proof of \theoref{unit-area}}
The proof follows from three key propositions.  The first is a combinatorial equivalence.
\begin{prop}\proplab{unit-area-circuit}
A $\Z^2$-colored graph $(G,\bgamma)$ is unit-area-Laman if and only if it is colored-Laman-sparse and
has $n$ vertices, $2n$ edges, and no subgraph with rank $2$ $\rho$-image and \eqref{z2-sparsity} holding with equality.
\end{prop}
\begin{proof}[Proof of \propref{unit-area-circuit}]
Comparing \eqref{z2-sparsity} with \eqref{unit-area-sparsity1}--\eqref{unit-area-sparsity2},
we see that unit-area-Laman graphs are exactly those which, after following the construction used to prove
\propref{lattice-fix-implies-rk2-block}, become colored-Laman.
\end{proof}

\paragraph{The Maxwell direction}
For the geometric part of the proof, we first derive the rigidity matrix.
If $\vec L = \ttmat{a}{b}{c}{d}$,
and we coordinatize infinitesimal motions as $(\vec v,\vec M)$ with $\vec M = \ttmat{p}{q}{r}{s}$,
then this has the form of \eqref{rigidity-matrix} plus one additional row corresponding to the equation
\begin{equation}\eqlab{unit-area-row}
\iprod{(d,-c,-b,a)}{(p,q,r,s)} = 0
\end{equation}
Violations of unit-area-Laman-sparsity come in two types, according to the rank  $k$ of the $\rho$-image.
For $k=0,1$, these are all violations of colored-Laman sparsity, implying, by \theoref{bigtheorem},
a generic dependency in the unit-area rigidity matrix that does not involve the row \eqref{unit-area-row}.
For $k=2$, \propref{unit-area-circuit} implies a new type of violation: a subgraph $(G',\bgamma)$
with $n'$ vertices, $\rho$-image of rank $2$, and $f(G')$ edges.  If such a subgraph forces a
generic periodic framework to fix the lattice representation $\vec L$, then the equation \eqref{unit-area-row}
is dependent on the equations corresponding to edge lengths.
The Maxwell direction then follows from the converse of \propref{lattice-fix-implies-rk2-block}.
\begin{prop}\proplab{unit-area-maxwell}
Let $(G,\bgamma)$ be a colored-Laman sparse graph with $\rho$-image of rank $2$ and
\eqref{z2-sparsity} met with equality.  Then an associated generic framework has only motions
that act trivially on the  $\Z^2$-representation $\vec L$.
\end{prop}
\begin{proof}[Proof of \propref{unit-area-maxwell}]
Let $(G,\bgamma)$ have $n$ vertices, and $c$ connected components.
It is sufficient to consider $(G,\bgamma)$ that is minimal with respect to the hypotheses of the
proposition, which forces every connected component to have $\rho$-image with rank
at least one.  In this case, there are $m = 2n + 3 - 2c$ edges.  By \theoref{bigtheorem},
the kernel of the rigidity matrix has dimension $2n + 4 - m = 2c + 1$.
Since the connected components can translate independently, and
the whole framework can rotate, there are at least $2c + 1$ dimensions of infinitesimal
motions acting trivially on the lattice.
\end{proof}

\paragraph{The Laman direction}
Now let $(G,\bgamma)$ be a unit-area-Laman graph.
\theoref{bigtheorem} implies that any generic periodic framework on $(G,\bgamma)$
has a $4$-dimensional space of infinitesimal motions, and that any non-trivial
infinitesimal motion is a linear combination of $3$ trivial ones and some other
infinitesimal motion $(\vec v,\vec M)$.  Since the trivial infinitesimal motions
act trivially on the lattice representation $\vec L$, if $(\vec v,\vec M)$
does as well, then a generic framework on $(G,\bgamma)$ fixes the lattice
representation.  By Propositions  \proprefX{unit-area-circuit} and
\proprefX{lattice-fix-implies-rk2-block} this is impossible, implying that
$(\vec v,\vec M)$ does not act trivially on the lattice representation.
However, it might preserve the area of the fundamental domain,
which would make \eqref{unit-area-row} part of a dependency in the unit-area rigidity matrix.
The Laman direction will then follow once we can exhibit a generic periodic
framework on $(G,\bgamma)$ for which $(\vec v,\vec M)$ does not preserve the
area of the fundamental domain.

To do this, we recall, from \propref{infinitesimal-rigidity-fact}, that a generic
linear transformation
\begin{equation}
\vec A = \ttmat{a}{b}{c}{d}
\end{equation}
preserves infinitesimal rigidity and
sends the non-trivial infinitesimal motion $(\vec v,\vec M)$ to another
non-trivial infinitesimal motion $(\vec v',\vec M')$, which is given by
\begin{eqnarray*}
\vec v_i' = \vec A^*\cdot \vec v_i & \qquad\text{for all $i\in V(G)$} \\
\vec M' = \vec A^*\cdot \vec M\eqlab{m-prime}
\end{eqnarray*}
where
\begin{equation}
\vec A^* = \det(\vec A)^{-1}\ttmat{d}{-c}{-b}{a}
\end{equation}
is the transpose of $\vec A^{-1}$.  The main step is this next proposition which says
that satisfying \eqref{unit-area-row} is \emph{not} affinely invariant.
\begin{prop}\proplab{unit-area-laman}
Let $(G,\bgamma)$ be a unit-area-Laman graph, and let $\tilde{G}(\vec p,\vec L)$
be a generic realization with $\vec L$ being the identity matrix, let $\vec A$
be a generic linear transformation, and let the infinitesimal motions
$(\vec v,\vec M)$ and $(\vec v',\vec M')$ be defined as above.  If
$(\vec v,\vec M)$ preserves the area of the fundamental domain, then
$(\vec v',\vec M')$ does not.
\end{prop}
\begin{proof}[Proof of \propref{unit-area-laman}]
Because $\vec L$ is the identity, $\vec M$ has the form
\begin{equation}\eqlab{m-form}
\vec M = \ttmat{\lambda}{\mu}{\nu}{-\lambda}
\end{equation}
either by direction computation, or by observing that it is an element of the Lie algebra $\operatorname{sl}(2)$,
as discussed above, we know that $-\mu\neq \nu$, since $(\vec v,\vec M)$ does not act trivially on $\vec L$.
In particular, $\mu$ and $\nu$ are not both zero.
Plugging in to \eqref{m-prime} we get
\begin{equation}
\vec M' =
\det(\vec A)^{-1}\ttmat{d\lambda - c\nu}{c\lambda+d\mu}{a\nu - b\lambda}{-a\lambda - b\mu}
\end{equation}
Plugging entries of $\vec M'$ in to the l.h.s. of \eqref{unit-area-row} to obtain:
\begin{equation}\eqlab{lhs-form}
\det(\vec A)^{-1}\left(\lambda(d^2 + b^2 - a^2-c^2) - (\mu+\nu)(ab+cd) \right)
\end{equation}
which is generically non-zero in the entries of $\vec A$: the conditions for \eqref{lhs-form}
are that its columns are the same length and orthogonal to each other.
\end{proof}
We now observe that, by \propref{infinitesimal-rigidity-fact},
there is a generic realization  $\tilde{G}(\vec p,\vec L)$
of a framework with unit-area-Laman colored quotient $(G,\bgamma)$ in which $\vec L$ is the identity.  If
the non-trivial infinitesimal motion $(\vec v,\vec M)$ does not satisfy \eqref{unit-area-row}, we are done.
Otherwise, the hypothesis of \propref{unit-area-laman} are met, and, thus, after applying a generic
linear transformation, the proof is complete.
\eop

\subsection{Fixed-lattice frameworks}
Another restricted class of periodic frameworks,
are \emph{fixed-lattice frameworks}. These are periodic frameworks, with the restriction
that the allowed motions act trivially on the lattice representation.  This model was
introduced by Whiteley \cite{W88} in the first investigation of generic rigidity with
forced symmetry.  More recenty, Ross discovered the
following%
\footnote{The sparsity counts we describe here are slightly different from what is stated in
\cite[Theorem 4.2.1]{R11}, but they are equivalent by an argument similar to that
in the proof of \propref{z2-sparsity-equiv}.  This presentation highlights the connection to
colored-Laman graphs.}  %
complete characterization of minimal rigidity for fixed-lattice frameworks.
\begin{theorem}[\fixedlatticetheorem][{\cite{R09}\cite[Theorem 4.2.1]{R11}}]\theolab{fixed-lattice}
Let $\tilde{G}(\vec p)$ be a generic fixed-lattice framework.  Then $\tilde{G}(\vec p)$ is minimally
rigid if and only if the associated colored graph $(G,\bgamma)$ has $n$ vertices, $m = 2n - 2$
edges and, for all subgraphs $G'$ of $G$ with $n'$ vertices, $m'$ edges, $c'_0$ connected components
with trivial $\rho$-image, and $c'_{\ge 1}$ connected components with non-trivial $\rho$-image:
\begin{equation}
m' \le 2n' - 3c'_0 -2c'_{\ge 1}
\end{equation}
\end{theorem}
We define the family of graphs appearing in \theoref{fixed-lattice} to be \emph{Ross graphs}.
In \cite{MT10}, we gave an alternate proof based on \theoref{bigtheorem}.  The
two steps are similar to the ones used to prove \theoref{unit-area}, except we
can take a ``shortcut'' in the argument by simulating fixing the lattice by adding
self-loops to the colored graph.  The geometric step is:
\begin{theorem}[\fixlatticemt][{\cite[Section 19.1]{MT10}}]\theolab{fixed-lattice-2}
Let $\tilde{G}(\vec p)$ be a generic fixed-lattice framework.  Then $\tilde{G}(\vec p)$ is
minimally rigid if and only if the associated colored graph $(G,\bgamma)$ plus three
self-loops colored $(1,0)$, $(0,1)$, $(1,1)$ added to any vertex is colored-Laman.
\end{theorem}
\theoref{fixed-lattice} then follows from the following combinatorial statement
that generalizes an idea of Lovász-Yemini \cite{LY82} and Recski \cite{R84}
(cf. \cite{HLST07,H02a}).
\begin{prop}[\rossbyadding][{\cite[Lemma 19.1]{MT10}}]\proplab{ross-by-adding}
A colored graph $(G,\bgamma)$ is a Ross graph if and only if adding three self-loops
colored $(1,0)$, $(0,1)$, $(1,1)$ to any vertex results in a colored-Laman graph.
\end{prop}

\subsection{Further connections}
Theorems \theorefX{cylinder} and \theorefX{unit-area} suggest a more general methodology
for obtaining Maxwell-Laman-type theorems for restrictions of periodic frameworks:
\begin{itemize}
\item Add an equation that restricts the allowed lattice representations $\vec L$.
\item Identify which generic periodic frameworks are the maximal ones that do
not imply the new restriction.
\end{itemize}
Our proof of \theoref{fixed-lattice} works this way as well:
adding self-loops adds three equations constraining the lattice representation.
Another perspective is that we are enlarging the class of trivial infinitesimal
motions by forcing one or more vectors into the kernel of the periodic rigidity matrix.
The most general form of this operation is known as the ``Elementary Quotient'' or
``Dilworth Truncation'', and it preserves \emph{representability}
of $(k,\ell)$-sparsity matroids \cite{ST11}, but obtaining \emph{rigidity} results (e.g., \cite{LY82})
requires geometric analysis specific to each case.  This section gives
a family of examples where we find new \emph{rigidity} matroids from each
other using a specialized version of Dilworth Truncation.

Ross \cite[Section 5]{R11} has studied some restrictions of periodic frameworks
as generalizations of the fixed-lattice model.  In this section we close the
circle of ideas, showing how to study them as specializations of the periodic one.

\subsection{One more variant}
We end this section with one more variation of the periodic model.  A \emph{fixed-angle framework} is defined to
be a periodic framework where the allowed motions preserve the angle between the
sides of the fundamental domain.
\begin{theorem}
A generic fixed angle framework is minimally rigid if and only if its associated
colored graph is unit-area-Laman.
\end{theorem}
\begin{proof}[Proof Sketch.]
The steps are similar to the proof of \theoref{unit-area}.
The new
row in the rigidity matrix corresponds to (in the same notation),
the partial derivatives of the equation:
\begin{equation}
\iprod{\frac{(a,c)}{||(a,c)||}}{\frac{(b,d)}{||(b,d)||}} = \text{const}
\end{equation}
so the new row in the rigidity matrix corresponds to:
\begin{equation}\eqlab{fixed-angle-row}
\iprod{\det(\vec L)\left(
c||(b,d)||^2, -d||(a,c)||^2,a||(b,d)||^2, -b||(a,c)||^2
\right)}{(p,q,r,s)} = 0
\end{equation}
The Maxwell direction is has a proof that is exactly the same as for \theoref{unit-area}.
For the Laman direction, we again start with a generic framework where $\vec L$ is the identity.  If
the non-trivial infinitesimal motion $(\vec v,\vec M)$ does not preserve \eqref{fixed-angle-row},
then we are done.  Otherwise, $\vec M$ has the form
\begin{equation}
\vec M = \ttmat{\mu}{\lambda}{-\lambda}{\nu}
\end{equation}
with $\mu$ and $\nu$ not both zero, because $\vec M$ does not act trivially on $\vec L$.
We then construct a new generic framework by applying a linear map
\begin{equation}
\vec A = \ttmat{1}{b}{0}{d}
\end{equation}
A computation, similar to that for \eqref{lhs-form} yields
\begin{equation}
-\frac{b \left(b^2 \mu +d^2 \mu +\nu \right)}{||(b,d)||^3}
\end{equation}
which is, generically, not zero.
\end{proof}

\section{Cone and reflection frameworks}\seclab{cone}
The next cases of \theoref{bigtheorem} are those of $\Z/2\Z$ acting on the plane by a single reflection
and $\Z/k\Z$ acting on the plane by rotation through angle $2\pi/k$.  The
sparsity invariants are particularly easy to characterize in these two cases:
\begin{itemize}
\item The Teichmüller space is empty, since any rotation center or reflection axis can be
moved on to another by an isometry.
\item The centralizer is all of $\Euc(2)$ for the trivial subgroup and otherwise consists of
rotation around a fixed center or translation parallel to the reflection axis.
\end{itemize}
Since the rank $k$ of the $\rho$-image of a $\Z/k\Z$-colored graph is always zero or one, we specialize \eqref{sparsity} to
obtain the sparsity condition for subgraphs on $n'$ vertices, $m'$ edges, $c'_0$ connected components with
trivial $\rho$-image and $c'_1$ connected components with non-trivial $\rho$-image.
\begin{equation}\eqlab{z-sparsity}
m' \le 2n' - 3c'_{0} - c'_{1}
\end{equation}
We define the family of $\Z$-colored graphs $(G,\bgamma)$ corresponding to minimally rigid
frameworks to be \emph{cone-Laman graphs}.  The name \emph{cone-Laman} comes from considering the
quotient of the plane by a rotation through angle $2\pi/k$, which is a flat cone, as shown in
\figref{cone}. Cone-Laman graphs are closely related to $(2,1)$-sparse graphs \cite{LS08}, and in
this section we use some sparse graph machinery to obtain combinatorial results on them.
\begin{figure}[htbp]
\centering
\subfigure[{A cone-Laman graph.}]{\includegraphics[width=.3\textwidth]{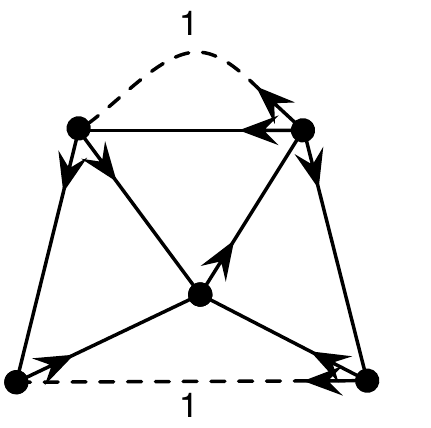}}
\subfigure[{A realization of (a) as a framework in a cone with opening angle $2\pi/3$.}]{\includegraphics[width=.4\textwidth]{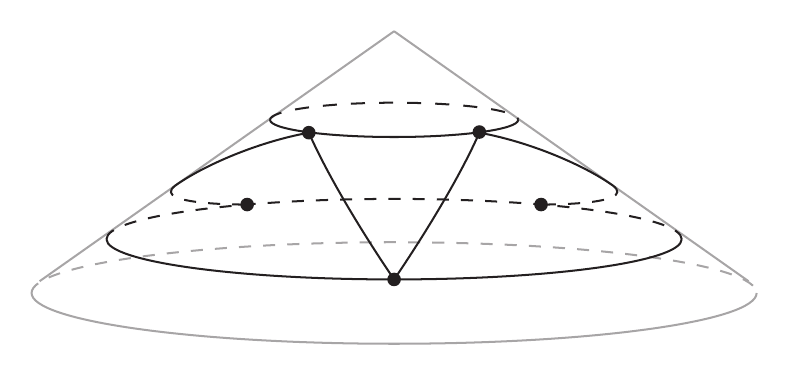}}
\label{fig:cone}
\caption{Figures from \cite{BHMT11}.}
\end{figure}
\subsection{Some background in $(k,\ell)$-sparse graphs}
In this section, we relate cone-Laman graphs to Laman graphs, and we will
repeatedly appeal to some standard results about $(k,\ell)$-sparse
graphs from \cite{LS08}.  In addition, we will require:
\begin{prop}\proplab{circuit-structure2}
Let $G$ be a $(2,1)$-graph.  If there is exactly one $(2,2)$-circuit in $G$, then $G$ is $(2,2)$-spanning.
Otherwise, $G$ is not $(2,2)$-spanning and the $(2,2)$-circuits in $G$ are vertex-disjoint.
\end{prop}
\begin{proof}
Let $G$ have $n$ vertices.  First assume that $G$ has exactly one $(2,2)$-circuit.
Then $G$ is a $(2,2)$-sparse graph $G'$ plus one edge; since $G$ has $2n-1$ edges,
$G'$ is a $(2,2)$-graph.  Otherwise there is more than one $(2,2)$-circuit.  Pick a
$(2,2)$-basis $G'$ of $G$.  In this case $G'$ does not have enough edges to be a $(2,2)$-graph,
so it decomposes, by \cite[Theorem 5]{LS08}, into vertex-disjoint $(2,2)$-components that span
all of the edges in $G\setminus G'$.  Because $G$ is $(2,1)$-sparse, it follows that each
$(2,2)$-component spans at most one edge of $G\setminus G'$, and thus at most one $(2,2)$-circuit.
We have now shown that the vertex sets of $(2,2)$-circuits in $G$ are each contained in a $(2,2)$-component of $G'$, which is
vertex-disjoint from the others.
\end{proof}

\subsection{Cone-Laman vs. cylinder-Laman}
By comparing the cylinder-Laman counts \eqref{cylinder-sparsity} with the cone-Laman counts \eqref{z-sparsity},
we can see that every cylinder-Laman graph, interpreted as having $\Z/k\Z$ colors, is also cone-Laman for $k$ large enough.
However, the two classes are not equivalent.  One can see this geometrically by considering a colored graph with
two self-loops: these are evidently dependent in the cylinder, and independent in the cone.  The conclusion
is that the interplay between $\teichgamma(\cdot)$ and $\cent(\cdot)$, can yield two different, geometrically
interesting sparse colored families on $(2,1)$-graphs.  The combinatorial relation is:
\begin{theorem}\theolab{cone-v-cylinder}
A $\Z$-colored graph $(G,\bgamma)$ is cylinder-Laman if and only if it is cone-Laman when interpreted
as having $\Z/k\Z$-colors for a sufficiently large $k$ and $G$ is $(2,2)$-spanning.
\end{theorem}
\begin{proof}
The only difficult thing to check is that a cylinder-Laman graph $(G,\bgamma)$ is $(2,2)$-spanning.
Assuming that $G$ is not $(2,2)$-spanning, \propref{circuit-structure2} supplies two vertex-disjoint
$(2,1)$-blocks.  If the union spans $n'$ vertices, there are $2n' - 2$ edges,
which violates \eqref{cylinder-sparsity}.
\end{proof}

\subsection{Connections to symmetric finite frameworks}
The following theorem of Schulze is superficially similar to
\theoref{bigtheorem} for $k=3$:
\begin{theorem}[\schulzecthree][{\cite[Theorem 5.1]{S10a}}]\theolab{schulze-c3}
Let $G$ be a Laman-graph with a free $\Z/3\Z$ action $\varphi$.  Then a generic
framework embedded such that $\varphi$ is realized by a rotation through angle
$2\pi/3$ is minimally rigid.
\end{theorem}
We highlight this result to draw a distinction between forced and incidental symmetry:
while \theoref{schulze-c3} is related to \theoref{bigtheorem}, it is not implied
by it.  The issue is that while infinitesimal motions of the cone-framework lift
to infinitesimal motions, only \emph{symmetric} infinitesimal motions of the lift
project to infinitesimal motions of the associated cone-framework.  Thus, from
\theoref{bigtheorem}, we learn that the lift of a generic minimally rigid cone
framework for $k=3$ has no \emph{symmetric} infinitesimal motion as a finite framework,
but there may be a \emph{non-symmetric} motion induced by the added symmetry.
An interesting question is whether the natural generalization of Schulze's Theorem
holds:
\begin{question}\conjlab{symmetric-finite}
Let $k>3$, and let $(G,\varphi)$ be a graph with a free $\Z/k\Z$-action.  Are
generic frameworks with $Z/k\Z$-symmetry rigid if and only if
$G$ is Laman-spanning and its colored quotient is cone-Laman-spanning?
\end{question}
That $G$ must be Laman-spanning is clear.  On the other hand, the discussion above and \theoref{bigtheorem}
imply that to avoid a symmetric non-trivial infinitesimal motion, a generic  $\Z/k\Z$-symmetric finite
framework must have cone-Laman-spanning quotient. Ross, Schulze and Whiteley \cite{RSW11} and
Schulze and Whiteley \cite{SW11} use this same idea in a number of interesting $3$-dimensional
applications.  The graphs described in the question are a family of simple $(2,0)$-graphs;
simple $(2,1)$-graphs have recently played a role in the theory of frameworks restricted to lie in
surfaces embedded in $\R^3$ \cite{NO11,NOP10}.

\subsection{The lift of a cone-Laman graph}
The lift $(\tilde{G},\varphi)$, defined in \secref{lift}, of a $\Z/k\Z$-colored graph is itself
a finite graph $(G,\varphi)$ with a free action by $\Z/k\Z$.  For $k\ge3$ prime, cone-Laman
graphs have a close connection to Laman graphs.
\begin{prop}[\conelift][{\cite[Lemma 6]{BHMT11}}]\proplab{cone-lift}
Let $k\ge 3$ be prime.  A $\Z/k\Z$-colored graph is cone-Laman if and only if its lift $(G,\varphi)$ has
as its underlying graph a Laman-sparse graph $G$ with $\tilde{n}$ vertices and
$2\tilde{n} - k$ edges.
\end{prop}
As noted in \cite{BHMT11}, this statement is false for $k=2$, so while we can relax the
hypothesis somewhat at the expense of a more complicated statement, they cannot all
be removed.

Although it is simple, \propref{cone-lift} is surprisingly powerful, since it shows that
one can study cone-Laman graphs using all the combinatorial tools related to Laman
graphs. \propref{cone-lift} depends in a fundamental way on the fact that cone-Laman graphs
have $2n-1$ edges, and it doesn't have a naive generalization to colored-Laman or
unit-area Laman graphs.
\begin{question}
What are the $\Z^2$-colored graphs $(G,\bgamma)$ with the property that every \emph{finite}
subgraph of the periodic lift $(\tilde{G},\varphi)$ is Laman-sparse?
\end{question}
We expect that this should be a more general family than unit-area-Laman graphs.  On the other
hand, it has been observed by Guest and Hutchinson \cite{GH03} that the lift of a colored-Laman graph
is not Laman-sparse.

\section{Groups with rotations and translations} \seclab{crystal}

The final case of \theoref{bigtheorem} is that of crystallographic groups acting discretely and
cocompactly by translations and rotations.  It is a classical fact \cite{B11,B12} that
all such groups other than $\Z^2$ are semi-direct products of the form
$$\Gamma_k :=\Z^2 \rtimes \Z/k\Z $$
where $k = 2, 3, 4, 6$. The action on $\Z^2$ by the generator of $\Z/k\Z$ is given
by the following table.
\begin{center}
\begin{tabular}{|c|c|c|c|c|}
\hline  $k$ & $2$ & $3$ & $4$ & $6$ \\
\hline matrix & $\JMat{-1 & 0 \\ 0 & -1}$  & $\JMat{ 0 & -1 \\ 1 & -1} $ & $\JMat{0 & -1 \\ 1 & 0} $
&  $\JMat{0 & -1 \\ 1 & 1} $ \\
\hline
\end{tabular}
\end{center}

\subsection{The quantities $\teich(\cdot)$ and $\cent(\cdot)$ for subgroups}
For any discrete faithful representation $\Phi: \Gamma_k \to \Euc(2)$, it is a (non-obvious) fact that for any element $t \in \Z^2$ in $\Gamma_k$, the image $\Phi(t)$
is necessarily a translation, and for any $r \in \Gamma_k \setminus \Z^2$, the image $\Phi(r)$ is necessarily a rotatation.  Consequently, we respectively call such
elements of $\Gamma_k$ translations and rotations, and we call $\Trans(\Gamma_k) = \Z^2$ the translation subgroup of $\Gamma_k$.  For any subgroup
$\Gamma' < \Gamma_k$, its translation subgroup is $\Trans(\Gamma') = \Gamma' \cap \Trans(\Gamma_k)$.

Let $\Phi$ be a representation of $\Gamma_k$.
In the cases $k \neq 2$, we must have $\Phi(\Trans(\Gamma_k))$ preserved by an order $k$ rotation, and so the image of $\Trans(\Gamma_k)$ is
determined by the image of a single nontrivial $t \in \Trans(\Gamma_k)$.  Furthermore, by acting on $\Phi$ by a rotation in $\Euc(2)$, we can always obtain a new
representation $\Phi'$ such that $\Phi(t)$ has translation vector $(\lambda, 0)$ for some $\lambda \in \R$.  Consequently, we have shown the following.

\begin{prop} \proplab{teich346}
Let $\Gamma'$ be a subgroup of $\Gamma_k$ for $k = 3, 4, 6$.  Then, $\teich_{\Gamma_k}(\Trans(\Gamma')) = 1$ if $\Trans(\Gamma')$ is nontrivial and is $0$ otherwise.
\end{prop}

In the case of $k = 2$, it turns out that since order $2$ rotations preserve all lattices, this puts no constraint on how $\Phi$ embeds $\Trans(\Gamma_2)$.  Consequently, we have
$\teich$ values similar to the periodic case.

\begin{prop} \proplab{teich2}
Let $\Gamma'$ be a subgroup of $\Gamma_2$.  Then, $\teich_{\Gamma_k}(\Trans(\Gamma')) = \text{max} (2\ell-1, 0)$ where $\ell = \rk(\Trans(\Gamma'))$.
\end{prop}

The dimension of the centralizer, similarly, is concrete and computable.  If a subgroup contains a translation $t$, then $\Phi(t)$
commutes precisely with translations of $\Euc(2)$.  If a subgroup $\Gamma'$ of $\Gamma_k$ is a cyclic subgroup of rotations, then $\Phi(\Gamma')$ is a group of rotations
with the same rotation center, and it is easy to see that such a group commutes precisely with the ($1$-dimensional) subgroup in $\Euc(2)$ of rotations with that center.
Consequently, we obtain the following characterization of $\cent$.

\begin{prop} \proplab{cent2346}
Suppose that $\Gamma'$ is a subgroup of $\Gamma_k$.  Then,

$$\cent(\Gamma') = \left\{ \begin{array}{rcl} 3 & & \text{if $\Gamma'$ is trivial} \\
2 & & \text{if $\Gamma'$ contains only translations} \\
1 & & \text{if $\Gamma'$ contains only rotations} \\
0 & & \text{if $\Gamma'$ contains both rotations and translations}  \end{array} \right. $$
\end{prop}

\subsection{The quantities $\teich(\cdot)$ and $\cent(\cdot)$ for colored graphs}

For any $\Gamma_k$-colored graph $(G', \bgamma)$, we associate subgroups of $\Gamma_k$.
Suppose $G$ has components $G'_1, \dots, G'_c$ and choose base vertices $b_1, \dots, b_c$.  We set $\Gamma_i' = \rho(\pi_1(G'_i, b_i))$, and
$$\Trans(G) = \langle \Trans(\Gamma_1)', \Trans(\Gamma_2'), \dots, \Trans(\Gamma_c') \rangle $$

The $\Gamma$-Laman sparsity counts are defined in terms of $\teich(\Trans(G))$ and $\cent(\Gamma_i')$.  Since we chose base vertices $b_i$, one might
worry that these quantities are not well-defined.  However, changing the base vertex in $G_i$ has the effect of conjugating
$\Gamma_i'$.  In $\Gamma_k$, conjugates of translations are translations and conjugates of rotations are rotations, so $\cent(\cdot)$
is then well-defined
by \propref{cent2346}, and $\teich(\Trans(G))$ for $k=3,4,6$ by \propref{teich346}.
Indeed, for $k=3,4,6$, $\teich(\Trans(G)) = 1$ if any $\Trans(\Gamma_i')$ is nontrivial
and is $0$ otherwise.  In $\Gamma_2$, all translation subgroups are normal, so $\Trans(\Gamma_i')$ itself does not
depend on the choice of base vertex.

\subsection{Computing $\teich$ and $\cent$ for $\Gamma_2$-colored graphs}
A quick and simple algorithm exists to compute $\teich(\Trans(G))$ and $\cent(\Gamma_i')$ which relies on finding a suitable generating set for $\Gamma_i'$.
A generating set for $\pi_1(G'_i, b_i)$ can be constructed as follows.  Find a spanning tree $T_i$ of component $G_i$.  Then for each edge $jk \in G_i - T_i$,
let $P_{jk}$ be the path traversing the (unique) path $b_i$ to $j$ in $T_i$, then $jk$, and then the (unique) path $k$ to $b_i$ in $T_i$.  The $P_{jk}$ ranging over $jk \in G_i - T_i$
generate $\pi_1(G'_i, b_i)$, and so $\eta_{jk} := \rho(P_{jk})$ ranging over the same set generates $\Gamma_i'$.

Next, relabel the generators of $\Gamma_i'$ as $r_{j, i}, t_{j, i}$ where the $r_{j, i}$ are rotations and the $t_{j, i}$ are translations.  If there are only translations, no modifications are
required and $\Trans(\Gamma_i') = \Gamma_i'$.  Otherwise, set $t'_{j, i} = r_{1, i} r_{j, i}$ for $j \geq 2$.  Since all rotations are order $2$, the $t'_{j, i}$ are all translations and
$\Gamma_i'$ is generated by $r_{1,i}$, the $t'_{j, i}$ and the $t_{j, i}$.  At this point, checking $\cent{\Gamma_i'}$ is straightforward.
Furthermore, one can show that $\Trans(\Gamma_i')$ is generated by the $t'_{j, i}$
and $t_{j, i}$, and so $\Trans(G)$ is generated by the $t'_{j, i}$ and the $t_{j, i}$ over all $i$ and $j$.  Then,
computing $\rk(\Trans(G))$ is basic linear algebra, and $\teich(\Trans(G))$ is given by \propref{teich2}.

\subsection{Computing $\teich$ and $\cent$ for $\Gamma_k$-colored graphs for $k \neq 2$}
In this case, all that needs to be determined is whether each $\Gamma_i'$ contains rotations, translations, or both.
Compute generators $r_{j, i}, t_{j, i}$ for $\Gamma_i'$ as above.  Then, $\Gamma_i'$ contains a rotation if
and only if there is at least one $r_{j, i}$.  The only real difficulty is determining if $\Gamma_i'$
contains translations when the generators are all rotations.  Any group consisting entirely of
rotations is cyclic (see, e.g., \cite[Lemma 4.2]{MT11}), and so it suffices to compute the commutators
$r_{1, i} r_{j, i} r_{1, i}^{-1} r_{j, i}^{-1}$ for $j \geq 2$.  The group contains no translations
if and only if these commutators are all trivial.

\bibliographystyle{plainnat}

\end{document}